\documentclass[amssymb,12pt]{article}

\usepackage{mathtools}
\usepackage{geometry,amssymb,amsthm,amsmath,
graphics,enumerate}
\usepackage{times}
\usepackage{amsfonts}
\usepackage{amscd}
\usepackage{url}

\author{ Michael C.\
Laskowski  
and Douglas S.\ Ulrich
 \thanks{Both authors partially supported
by NSF grant DMS-1855789.}
\\
Department of Mathematics\\University of Maryland
}

\def\includeE#1{{\lhook\kern-3.5pt\joinrel\smash{
    \mathop{\longrightarrow}\limits^{#1}}}}

\def\efor/{Example~\ref{E4}}

\def\BL/{Baldwin--Lachlan}
\def\Bu/{Buechler}
\def\Hr/{Hrushovski}
\def\lm/{locally modular}
\def\wm/{weakly minimal}
\def\nm/{non--modular}
\def\ss/{superstable}
\def\ud/{unidimensional}
\def\sm/{strongly minimal}

\def\abar{\overline{a}}

\def\hbar{\overline{h}}

\def\stp{{\rm stp}}

\def\tr/{trivial}
\def\nt/{non--trivial}
\def\st/{strong type}
\def\conc{{\char'136}}
\def\abar{\bar{a}}

\def\phi{\varphi}

\def\B{{\cal B}}

\def\FF{{\bf F}}

\def\stp{{\rm stp}}

\def\Fa0{{\FF^a_{\aleph_0}}}

\def\<{\langle}
\def\>{\rangle}

\newtheorem{Theorem}{Theorem}[section]
\newtheorem{Proposition}[Theorem]{Proposition}
\newtheorem{Definition}[Theorem]{Definition}

\newtheorem{Remark}[Theorem]{Remark}

\newtheorem{Lemma}[Theorem]{Lemma}
\newtheorem{Corollary}[Theorem]{Corollary}

\def\Mod{{\rm Mod}}
\def\iso{\cong}
\def\G{{\cal G}}
\def\H{{\cal H}}
\def\V{{\mathbb V}}

    \newcommand\myrestriction{\mathord\restriction}
\def\mr#1{\myrestriction_{#1}}

\begin{document}

	\title{Most(?) theories have Borel complete reducts}
	
	\date{\today} 
	
	\maketitle

\begin{abstract}	  We prove that many seemingly simple theories have Borel complete reducts.  Specifically, if a countable theory  has uncountably many complete 1-types,
then it has a Borel complete reduct.  Similarly, if $Th(M)$ is not small, then $M^{eq}$ has a Borel complete reduct, and if a theory $T$  is not $\omega$-stable, then 
the elementary diagram of  some countable model of $T$ has a Borel complete reduct.
	\end{abstract}

\section{Introduction}	
In their seminal paper \cite{FS}, Friedman and Stanley define and develop a notion of {\em Borel reducibility} among classes of structures with universe $\omega$ in a fixed, countable language $L$ that are Borel and invariant under permutations of $\omega$.  It is well known (see e.g., \cite{KechrisDST} or \cite{GaoIDST}) that such classes are of the form $\Mod(\Phi)$, the set of models of $\Phi$ whose universe is precisely $\omega$ for some sentence $\Phi\in L_{\omega_1,\omega}$, but here we concentrate on
first-order, countable theories $T$.   For countable theories $T,S$ in possibly different language, a  {\em Borel reduction} is a Borel function $f:\Mod(T)\rightarrow \Mod(S)$ that satisfies $M\iso N$ if and only if $f(M)\iso f(N)$.  One says that $T$ is {\em Borel reducible} to $S$    if there is a Borel reduction $f:\Mod(T)\rightarrow \Mod(S)$.
As Borel reducibility is transitive, this induces a quasi-order on the class of all countable theories, where we say $T$ and $S$ are
 {\em Borel equivalent} if there are Borel reductions in both directions. 
In \cite{FS}, Friedman and Stanley show that among Borel invariant classes (hence among countable first-order theories) there is a maximal class with respect to $\le_B$.  We say $\Phi$ is {\em Borel complete} if it is in this maximal class.  Examples include the theories of graphs, linear orders, groups, and fields.

The intuition is that Borel complexity of a theory $T$ is related to the complexity of invariants that describe the isomorphism types of countable models of $T$.  
Given an $L$-structure $M$, one naturally thinks of the reducts $M_0$ of $M$ to be `simpler objects' hence the invariants for a reduct `should' be no more complicated than
for the original $M$, but we will see that this intuition is incorrect.  As a paradigm, let $T$ be the theory of `independent unary predicates' i.e., $T=Th(2^\omega,U_n)$, where
each $U_n$ is a unary predicate interpreted as $U_n=\{\eta\in 2^\omega:\eta(n)=1\}$.  The countable models of $T$ are rather easy to describe.  
The isomorphism type of a model is specified by which countable, dense subset of `branches' is realized, and  how many elements realize each of those branches.  However, with Theorem~\ref{1types},  we will see that $T$ has a Borel complete reduct.

To be precise about reducts, we have the following definition.

\begin{Definition} \label{reduct} {\em  Given an $L$-structure $M$, a {\em reduct $M'$ of $M$} is an $L'$-structure with the same universe as $M$, and for which the interpretation every atomic $L'$-formula $\alpha(x_1,\dots,x_k)$ is an $L$-definable subset of $M^k$ (without parameters). 
An $L'$-theory $T'$ is a {\em reduct of an $L$-theory $T$} if $T'=Th(M')$ for some reduct $M'$ of some model $M$ of $T$.}
\end{Definition}
In the above definition, it would be equivalent to require that the interpretation in $M'$ of every $L'$-formula $\theta(x_1,\dots,x_k)$ is a 0-definable subset of $M^k$.

\section{An engine for  Borel completeness results}

This section is devoted to proving  Borel completeness for a specific family of theories.  All of the theories  $T_h$, are in the same language  $L=\{E_n:n\in\omega\}$
and are indexed by strictly increasing functions $h:\omega\rightarrow\omega\setminus\{0\}$.  For a specific choice of $h$, the theory $T_h$ asserts that
\begin{itemize}
\item Each $E_n$ is an equivalence relation with exactly $h(n)$ classes; and
\item  The $E_n$'s cross-cut, i.e., for all nonempty, finite $F\subseteq\omega$, $E_F(x,y):=\bigwedge_{n\in F} E_n(x,y)$ is an equivalence relation with
precisely $\Pi_{n\in F} h(n)$ classes.
\end{itemize}

It is well known that each of these theories $T_h$ is complete and admits elimination of quantifiers.  Thus, in any model of $T_h$,
there is a unique 1-type.  However, the strong type structure is complicated.\footnote{Recall that in any structure $M$,
two elements $a,b$ have the same {\em strong type}, $\stp(a)=\stp(b)$, if $M\models E(a,b)$ for every 0-definable equivalence relation.
Because of the quantifier elimination, in any model $M\models T_h$, $\stp(a)=\stp(b)$ if and only if $M\models E_n(a,b)$ for every $n\in\omega$.} 
So much so, that the whole of this section is devoted to the proof of:

\begin{Theorem}  \label{big}  For any strictly increasing $h:\omega\rightarrow\omega\setminus\{0\}$, $T_h$ is Borel complete.
\end{Theorem}

\begin{proof}  Fix a strictly increasing function $h:\omega\rightarrow\omega\setminus\{0\}$.  We begin by describing representatives $\B$ of the strong types
and a group $G$ that acts faithfully and transitively on $\B$.
As notation, for each $n$, let $[h(n)]$ denote the $h(n)$-element set $\{1,\dots,h(n)\}$ and let $Sym([h(n])$ be the (finite) group of permutations of $[h(n)]$.  
Let $$\B=\{f:\omega\rightarrow\omega:f(n)\in[h(n)] \ \hbox{for all $n\in\omega\}$}$$
and let $G=\Pi_{n\in\omega} Sym([h(n)])$ be the direct product. As notation, for each $n\in\omega$, let $\pi_n:G\rightarrow Sym([h(n)])$ be the natural projection map.
Note that $G$ acts coordinate-wise on $\B$ by: For $g\in G$ and $f\in \B$, $g\cdot f$ is the element of $\B$ satisfying
$g\cdot f(n)=\pi_n(g)(f(n))$.

Define an equivalence relation $\sim$ on $\B$ by:
$$f\sim f'\quad \hbox{if and only if} \quad \{n\in\omega:f(n)\neq f'(n)\}\ \hbox{is finite.}$$
For $f\in\B$, let $[f]$ denote the $\sim$-class of $f$ and, abusing notation somewhat, for $W\subseteq\B$
$$[W]:=\bigcup\{[f]:f\in W\}.$$ 
Observe that for every $g\in G$, the permutation of $\B$ induced by the action of $g$ maps $\sim$-classes onto $\sim$-classes, i.e., $G$ also acts transitively on $\B/\sim$.

We first identify a countable family of $\sim$-classes that are `sufficiently indiscernible'.    Our first lemma is where we use the fact that the function $h$ defining
$T_h$ is strictly increasing.

\begin{Lemma}  \label{Claim1}  There is a countable set $Y=\{f_i:i\in\omega\}\subseteq \B$ such that whenever $i\neq j$, $\{n\in\omega: f_i(n)=f_j(n)\}$ is finite.
\end{Lemma}

\begin{proof}  We recursively construct $Y$ in $\omega$ steps.  Suppose $\{f_i:i<k\}$ have been chosen.  Choose an integer $N$ large enough so that $h(N)>k$ (hence $h(n)>k$ for all
$n\ge N$).  Now, construct $f_k\in\B$ to satisfy $f_k(n)\neq f_i(n)$ for all $n\ge N$ and all $i<k$.
\end{proof}

{\em Fix an enumeration $\<f_i:i\in\omega\>$ of  $Y$ for the whole of the argument.}  The `indiscernibility' of $Y$ alluded to above is formalized by the following definition and lemma.

\begin{Definition}  \label{respect} {\em  Given a permutation $\sigma\in Sym(\omega)$, a group element $g\in G$ {\em respects $\sigma$} if
$g\cdot [f_i]=[f_{\sigma(i)}]$ for every $i\in\omega$.}
\end{Definition}

\begin{Lemma} \label{Claim2}  For every permutation $\sigma\in Sym(\omega)$, there is some $g\in G$ respecting $\sigma$.
\end{Lemma}

\begin{proof}  Note that since $h$ is increasing, $h(n)\ge n$ for every $n\in\omega$.  Fix a permutation $\sigma\in Sym(\omega)$ and we will
define  some $g\in G$ respecting $\sigma$ coordinate-wise.  
Using Lemma~\ref{Claim1}, choose a sequence
$$0=N_0\ll N_1\ll N_2\ll \dots$$ of integers such that for all $i\in\omega$, both $f_i(n)\neq f_j(n)$ and $f_{\sigma(i)}(n)\neq f_{\sigma(j)}(n)$ hold for all $n\ge N_i$ and
all $j<i$.

Since $\{N_i\}$ are increasing, it follows that for each $i\in\omega$ and all $n\ge N_i$, the subsets $\{f_j(n):j\le i\}$ and $\{f_{\sigma(j)}(n):j\le i\}$ of $[h(n)]$
each have precisely $(i+1)$ elements.  Thus, for each $i<\omega$ and for each $n\ge N_i$, there is a permutation $\delta_n\in Sym([h(n)])$ satisfying
$$\bigwedge_{j\le i} \delta_n(f_j(n))=f_{\sigma(j)}(n)$$
[Simply begin defining $\delta_n$ to meet these constraints, and then complete $\delta_n$ to a permutation of $[h(n)]$ arbitrarily.]
Using this, define $g:=\<\delta_n:n\in\omega\>$, where  each $\delta_n\in Sym([h(n)])$ is constructed as  above.
To see that $g$ respects $\sigma$, note that for every $i\in\omega$, $(g\cdot f_i)(n)=f_{\sigma(i)}(n)$ for all $n\ge N_{i}$, so $(g\cdot f_i)\sim f_{\sigma(i)}$.
\end{proof}

\begin{Definition}  {\em  For distinct  integers $i\neq j$, let $d_{i,j}\in\B$ be defined by:
$$d_{i,j}(n):=\begin{cases} f_i(n) & \text{if $n$ even;}
 \\  f_j(n) & \text{if $n$ odd.}\end{cases}$$
 Let $Z:=\{d_{i,j}:i\neq j\}$.
}
\end{Definition}

Note that  $d_{i,j}\not\sim f_k$ for all distinct $i,j$ and all $k\in\omega$, hence  $\{[f_i]:i\in\omega\}$ and $\{[d_{i,j}]:i\neq j\}$ are disjoint.

\begin{Lemma}  \label{Claim3}
  For all $\sigma\in Sym(\omega)$, if $g\in G$ respects $\sigma$, then  $g\cdot [d_{i,j}]=[d_{\sigma(i),\sigma(j)}]$
for all  $i\neq j$.
\end{Lemma}

\begin{proof}  Choose $\sigma\in Sym(\omega)$, $g$ respecting $\sigma$, and $i\neq j$.  Choose $N$ such that  $(g\cdot [f_i])(n)=[f_{\sigma(i)}](n)$ and 
$(g\cdot [f_j])(n)=[f_{\sigma(j)}](n)$ for every $n\ge N$.
Since  $d_{i,j}(n)=f_i(n)$ for $n\ge N$ even, 
$$(g\cdot d_{i,j})(n)=\pi_n(g)(d_{i,j}(n))=\pi_n(g)(f_i(n))=(g\cdot f_i)(n)=f_{\sigma(i)}(n)$$
Dually, $(g\cdot d_{i,j})(n)=f_{\sigma(j)}(n)$ when $n\ge N$ is odd, so $(g\cdot d_{i,j})\sim d_{\sigma(i),\sigma(j)}$.
\end{proof}

With the combinatorial preliminaries out of the way, we now prove that $T_h$ is Borel complete.  
We form a highly homogeneous model $M^*\models T_h$ and thereafter, all models we consider will be countable, elementary substructures of $M^*$.
Let
$A=\{a_f:f\in\B\}$ and $B=\{b_f:f\in\B\}$ be disjoint sets and let
$M^*$ be the $L$-structure with universe $A\cup B$ and each $E_n$ interpreted by the rules:
\begin{itemize}
\item  For all $f\in\B$ and $n\in\omega$, $E_n(a_f,b_f)$; and
\item  For all $f,f'\in\B$ and $n\in\omega$, $E_n(a_f,a_{f'})$ iff $f(n)=f'(n)$.
\end{itemize}
with the other instances of $E_n$ following by symmetry and transitivity. 
For any finite $F\subseteq\omega$, $\{f\mr{F}:f\in\B\}$ has exactly $\Pi_{n\in F} h(n)$ elements,
hence $E_F(x,y):=\bigwedge_{n\in F} E_n(x,y)$ has $\Pi_{n\in F} h(n)$ classes in $M^*$.
Thus, the $\{E_n:n\in\omega\}$ cross cut and
$M^*\models T_h$.

Let $E_\infty(x,y)$ denote the (type definable) equivalence relation $\bigwedge_{n\in\omega} E_n(x,y)$.
Then, in $M^*$, $E_\infty$ partitions $M^*$ into 2-element classes $\{a_f,b_f\}$, indexed by $f\in\B$.
Note also that every $g\in G$ induces an $L$-automorphism $g^*\in Aut(M^*)$ by
$$g^*(x):=\begin{cases} a_{(g\cdot f)} & \text{if $x=a_f$ for some $f\in\B$}\\  b_{(g\cdot f)} & \text{if $x=b_f$ for some $f\in\B$}\end{cases}$$

Recall the set $Y=\{f_i:i\in\B\}$ from Lemma~\ref{Claim1}, so $[Y]=\{[f_i]:i\in\omega\}$.
Let $M_0\subseteq M^*$ be the substructure with universe $\{a_f:f\in [Y]\}$.  As $T_h$ admits elimination of quantifiers and as $[Y]$ is dense
in $\B$,  $M_0\preceq M^*$.
Moreover, every substructure $M$ of $M^*$ with universe containing $M_0$ will also be an elementary substructure of $M^*$, hence a model of $T_h$.

To show that $Mod(T_h)$ is Borel complete, we define a Borel mapping from $\{$irreflexive graphs $\G=(\omega,R)\}$ to $Mod(T_h)$ as follows:
Given $\G$, let $Z(R):=\{d_{i,j}\in Z:\G\models R(i,j)\}$, so $[Z(R)]=\bigcup\{[d_{i,j}]:d_{i,j}\in Z(R)\}$. Let
$M_G\preceq M^*$ be the substructure with universe
$$M_0\cup\{a_d,b_d:d\in [Z(R)]\}$$
That the map $\G\mapsto M_G$ is Borel is routine, given that $Y$ and $Z$ are fixed throughout.

Note that in $M_G$, every $E_\infty$-class has
has either one or two elements.  Specifically,
for each $d\in [Z(R)]$,  the $E_\infty$-class $[a_d]_\infty=\{a_d,b_d\}$,
while the $E_\infty$-class $[a_f]_\infty=\{a_f\}$ for every $f\in[Y]$.

We must show that for any two  graphs $\G=(\omega,R)$ and $\H=(\omega,S)$,
$\G$ and $\H$ are isomorphic if and only if the $L$-structures $M_G$ and $M_H$ are isomorphic. 

To verify this, first choose a graph isomorphism $\sigma:(\omega,R)\rightarrow (\omega,S)$.
Then $\sigma\in Sym(\omega)$ and, for distinct  integers $i\neq j$, $d_{i,j}\in Z(R)$ if and only if $d_{\sigma(i),\sigma(j)}\in Z(S)$.
Apply Lemma~\ref{Claim2}  to get $g\in G$ respecting $\sigma$ and let $g^*\in Aut(M^*)$ be the $L$-automorphism induced
by $g$.  By Lemma~\ref{Claim3} and Definition~\ref{respect}, it is easily checked that the restriction of $g^*$ to $M_G$
is an $L$-isomorphism between $M_G$ and $M_H$.

Conversely, assume that $\Psi:M_G\rightarrow M_H$ is an $L$-isomorphism.  Clearly, $\Psi$ maps $E_\infty$-classes in $M_G$ to $E_\infty$-classes in $M_H$.  In particular,
$\Psi$ permutes the 1-element $E_\infty$-classes $\{\{a_f\}:f\in [Y]\}$ of both $M_G$ and $M_H$, and maps  the 2-element $E_{\infty}$-classes $\{\{a_d,b_d\}:d\in [Z(R)]\}$  of $M_G$
onto the 2-element $E_{\infty}$-classes $\{\{a_d,b_d\}:d\in [Z(S)]\}$  of $M_H$.   That is, $\Psi$ induces a bijection $F:[Y\sqcup Z(R)]\rightarrow [Y\sqcup Z(S)]$ that permutes $[Y]$.

As well, by the interpretations of the $E_n$'s, for $f,f'\in [Y\sqcup Z(R)]$ and $n\in\omega$,
$$f(n)=f'(n)\quad\hbox{if and only if}\quad F(f)(n)=F(f')(n).$$
From this it follows that $F$ maps $\sim$-classes onto $\sim$-classes.  
As  $F$ permutes $[Y]$ and  as $[Y]=\bigcup\{[f_i]:i\in\omega\}$, $F$ induces a permutation $\sigma\in Sym(\omega)$ given by 
$\sigma(i)$ is the unique $i^*\in\omega$ such that $F([f_i])=[f_{i^*}]$.

We claim that this $\sigma$ induces a graph isomorphism between $\G=(\omega,R)$ and $\H=(\omega,S)$.  Indeed, choose any $(i,j)\in R$.
Thus, $d_{i,j}\in Z(R)$.  As $F$ is $\sim$-preserving, choose $N$ large enough so that $F(f_i)(n)=F(f_{\sigma(i)})(n)$ and
$F(f_j)(n)=F(f_{\sigma(j)})(n)$ for every $n\ge N$.  By definition of $d_{i,j}$,  $d_{i,j}(n)=f_i(n)$ for $n\ge N$ even, so
$F(d_{i,j})(n)=F(f_i)(n)=f_{\sigma(i)}(n)$ for such $n$.  Dually, for $n\ge N$ odd,  $F(d_{i,j})(n)=F(f_j)(n)=f_{\sigma(j)}(n)$.
Hence, $F(d_{i,j})\sim d_{\sigma(i),\sigma(j)}\in [Z(S)]$.  Thus, $(\sigma(i),\sigma(j))\in S$.  The converse direction is symmetric 
(i.e., use $\Psi^{-1}$ in place of $\Psi$ and run the same argument).
\end{proof}

\begin{Remark}  {\em If we relax the assumption that $h:\omega\rightarrow\omega\setminus\{0\}$ is strictly increasing, there are two cases.
If $h$ is unbounded, then the proof given above can easily be modified to show that the associated $T_h$ is also Borel complete.  Conversely, with Theorem~6.2 of
 \cite{LU} the authors
prove that if
$h:\omega\rightarrow\omega\setminus\{0\}$ is bounded, then $T_h$ is not Borel complete.   
The salient  distinction between the two cases is that when $h$ is bounded,  the associated group $G$ has bounded exponent.
However, even in the bounded case
$T_h$ has a Borel complete reduct by Lemma~\ref{cross} below.
}
\end{Remark}

\section{Applications to reducts}

We begin with one easy lemma that, when considering reducts, obviates the need for the number of classes to be strictly increasing.

\begin{Lemma}  \label{cross}  Let $L=\{E_n:n\in\omega\}$ and let  $f:\omega\rightarrow\omega\setminus\{0,1\}$ be any function.
Then every model $M$ of $T_f$,  the complete  theory asserting that each $E_n$ is an equivalence relation with $f(n)$  classes, and that the $\{E_n\}$ cross-cut,
has a Borel complete reduct.
\end{Lemma}

\begin{proof}  Given any function $f:\omega\rightarrow\omega\setminus\{0,1\}$, choose
 a partition $\omega=\bigsqcup\{F_n:n\in\omega\}$ into non-empty finite sets for which  $\Pi_{k\in F_n} f(k)< \Pi_{k\in F_m} f(k)$ whenever $n<m<\omega$.
 For each $n$, let $h(n):=\Pi_{k\in F_n} f(k)$ and let $E^*_n(x,y):=\bigwedge_{k\in F_n} E_k(x,y)$.  Then, as $h$ is strictly increasing
 and $\{E^*_n\}$ is a cross-cutting set of equivalence relations with each $E_n^*$ having $h(n)$ classes.  
 
 Now let $M\models T_f$ be arbitrary and let $L'=\{E^*_n:n\in\omega\}$. As each $E_n^*$ described above is 0-definable in $M$, 
 there is an $L'$-reduct $M'$ of $M$.  It follows from Theorem~\ref{big} that $T'=Th(M')$ is Borel complete, so $T_f$ has a Borel complete reduct.
\end{proof}

\begin{Theorem}  \label{1types} Suppose $T$ is a complete theory in a countable language with uncountably many 1-types.
Then every model $M$ of $T$ has a Borel complete reduct.
\end{Theorem}

\begin{proof}  Let $M\models T$ be arbitrary.
As usual, by the Cantor-Bendixon analysis of the compact, Hausdorff Stone space $S_1(T)$ of complete 1-types, choose a set
$\{\phi_\eta(x):\eta\in 2^{<\omega}\}$ of 0-definable formulas, indexed by the tree $(2^{<\omega},\trianglelefteq)$ ordered by initial segment,
satisfying:
\begin{enumerate}
\item  $M\models\exists x\phi_\eta(x)$ for each $\eta\in 2^{<\omega}$;
\item  For $\nu\trianglelefteq\eta$, $M\models\forall x(\phi_\eta(x)\rightarrow\phi_\nu(x))$;
\item  For each $n\in\omega$, $\{\phi_\eta(x):\eta\in 2^n\}$ are pairwise contradictory.
\end{enumerate}
By increasing these formulas slightly, we can additionally require
\begin{enumerate}\setcounter{enumi}{3}
\item  For each $n\in\omega$, $M\models\forall x( \bigvee_{\eta\in 2^n} \phi_\eta(x))$.
\end{enumerate}

Given such a tree of formulas, for each $n\in\omega$, define
$$\delta^0_n(x):=\bigwedge_{\eta\in 2^n} [\phi_\eta(x)\rightarrow\phi_{\eta\conc 0}(x)]\quad \hbox{and}\quad \delta^1_n(x):=\bigwedge_{\eta\in 2^n} [\phi_\eta(x)\rightarrow\phi_{\eta\conc 1}(x)]$$
Because of (4) above, $M\models\forall x(\delta^0_n(x)\vee\delta^1_n(x))$ for each $n$.  Also, for each $n$, let
$$E_n(x,y):=[\delta_n^0(x)\leftrightarrow \delta_n^0(y)]$$
From the above, each $E_n$ is a 0-definable equivalence relation with precisely two classes.

\medskip
\noindent{\bf Claim.}  The equivalence relations $\{E_n:n\in\omega\}$ are cross-cutting.

\begin{proof}  It suffices to prove that for every $m>0$,  the equivalence relation $E^*_m(x,y):=\bigwedge_{n<m} E_n(x,y)$ has $2^m$ classes.
So fix $m$ and choose a subset $A_m=\{a_\eta:\eta\in 2^m\}\subseteq M$ forming a set of representatives for the formulas $\{\phi_\eta(x):\eta\in 2^m\}$.
It suffices to show that $M\models\neg E^*_m(a_\eta,a_\nu)$ whenever $\eta\neq \nu$ are from $2^m$.
But this is clear.  Fix distinct $\eta\neq\nu$ and choose  any $k<m$   such that $\eta(k)\neq \nu(k)$.  Then $M\models \neg E_k(a_\eta,a_\nu)$, hence
$M\models\neg E^*_m(a_\eta,a_\nu)$. 
\end{proof}

Thus,  taking the 0-definable relations $\{E_n\}$, $M$ has a reduct that is a model of $T_f$ (where $f$ is the constant function 2).
As reducts of reducts are reducts, it follows from Lemma~\ref{cross} and Theorem~\ref{big} that $M$ has a Borel complete reduct.
\end{proof}

We highlight how unexpected Theorem~\ref{1types} is with two examples.
First, the theory of `Independent unary predicates' mentioned in the Introduction has a Borel complete reduct.

Next, we explore the assumption that a countable, complete theory $T$ is not small, i.e., for some $k$ there are uncountably many $k$-types.
We conjecture that some model of $T$ has a Borel complete reduct.
If $k=1$, then by Theorem~\ref{1types}, every model of $T$ has a Borel complete reduct.  If $k>1$ is least, then it is easily seen that there is some complete
$(k-1)$ type $p(x_1,\dots,x_{k-1})$ with uncountably many complete $q(x_1,\dots,x_k)$ extending $p$.  Thus, if $M$ is any model of $T$ realizing $p$, say by $\abar=(a_1,\dots,a_{k-1})$,
the expansion $(M,a_1,\dots,a_{k-1})$ has a Borel complete reduct, also by Theorem~\ref{1types}.   Similarly, we have the following result.

\begin{Corollary}  \label{small}  Suppose $T$ is a complete theory in a countable language that is not small.
Then for any model $M$ of $T$, $M^{eq}$ has a Borel complete reduct.
\end{Corollary}

\begin{proof}  Let $M$ be any model of $T$ and choose $k$ least such that $T$ has uncountably many complete $k$-types consistent with it.  
In the language $L^{eq}$, there is a sort $U_k$ and a definable bijection  $f:M^k\rightarrow U_k$.  Hence $Th(M^{eq})$ has uncountably many 1-types
consistent with it, each extending $U_k$.  Thus, $M^{eq}$ has a Borel complete reduct by Theorem~\ref{1types}.
\end{proof}

Finally, recall that a countable, complete theory is not $\omega$-stable if, for some countable model $M$ of $T$, the Stone space $S_1(M)$ is uncountable.
From this, we immediately obtain our final corollary.

\begin{Corollary}  \label{omegastable}  If a countable, complete $T$ is not $\omega$-stable, then for some countable model $M$ of $T$, the elementary diagram of
$M$ in the language $L(M)=L\cup\{c_m:m\in M\}$ has a Borel complete reduct.
\end{Corollary}

\begin{proof}  Choose a countable $M$ so that $S_1(M)$ is uncountable.  Then, in the language $L(M)$, the theory of the expanded structure $M_M$ in the language $L(M)$
has uncountably many 1-types, hence it has a Borel complete reduct by Theorem~\ref{1types}.
\end{proof}

The results above are by no means characterizations.  Indeed, there are many Borel complete $\omega$-stable theories.  In \cite{LSh}, the first author and Shelah
prove that any $\omega$-stable theory that has eni-DOP or is eni-deep is not only Borel complete, but also $\lambda$-Borel complete for all $\lambda$.\footnote{Definitions of 
eni-DOP and eni-deep are given in  Definitions~2.3 and 6.2, respectively, of \cite{LSh}, and
the definition of $\lambda$-Borel complete is recalled in Section~\ref{obs} of this paper.}
As well, there are $\omega$-stable theories with only countably many countable models that have Borel complete reducts.  To illustrate this, we introduce three interrelated
theories.  The first, $T_0$ in the language $L_0=\{U,V,W,R\}$ is the paradigmatic DOP theory.  $T_0$ asserts that:
\begin{itemize}
\item   $U,V,W$ partition the universe; 
\item $R\subseteq U\times V\times W$;
\item  $T_0\models \forall x\forall y\exists^\infty z R(x,y,z)$; [more formally, for each $n$, $T_0\models \forall x\forall y\exists^{\ge n} z R(x,y,z)$]; 
\item  $T_0\models \forall x\forall x'\forall y\forall y'\forall z [R(x,y,z)\wedge R(x',y',z)\rightarrow(x=x'\wedge y=y')]$.
\end{itemize}
$T_0$ is both $\omega$-stable and $\omega$-categorical and its unique countable model is rather tame.   
The  complexity of $T_0$ is only witnessed with uncountable models, where one can code arbitrary  bipartite graphs
in an uncountable model  $M$ by choosing the cardinalities of the sets $R(a,b,M)$ among $(a,b)\in U\times V$ to be either $\aleph_0$ or $|M|$.

To get bad behavior of countable models, we expand $T_0$ to an $L=L_0\cup\{f_n:n\in\omega\}$-theory $T\supseteq T_0$ that additionally asserts:
\begin{itemize}
\item  Each $f_n:U\times V\rightarrow W$;
\item  $\forall x\forall y R(x,y,f_n(x,y))$ for each $n$; and
\item  for distinct $n\neq m$, $\forall x\forall y (f_n(x,y)\neq f_m(x,y))$.
\end{itemize}
This $T$ is $\omega$-stable with eni-DOP and hence is Borel complete by Theorem~4.12 of \cite{LSh}.

However, $T$ has an expansion $T^*$ in a language $L^*:=L\cup\{c,d,g,h\}$ whose models are much better behaved.
Let $T^*$ additionally assert:
\begin{itemize}
\item  $U(c)\wedge V(d)$;
\item  $g:U\rightarrow V$ is a bijection with $g(c)=d$;  
\item  Letting $W^*:=\{z:R(c,d,z)\}$, $h:U\times V\times W^*\rightarrow W$ is an injective map that is the identity on $W^*$ and, for each $(x,y)\in U\times V$,
 maps $W^*$ onto $\{z\in W: R(x,y,z)\}$; and moreover
 \item  $h$ commutes with each $f_n$, i.e., $\forall x\forall y (h(x,y,f_n(c,d))=f_n(x,y))$.
 \end{itemize}
 Then $T^*$ is $\omega$-stable and two-dimensional (the dimensions being $|U|$ and $|W^*\setminus\{f_n(c,d):n\in\omega\}|$), hence $T^*$ has only countably many countable models.
 However, $T^*$ visibly has a Borel complete reduct, namely $T$.

\section{Observations about the theories $T_h$}   \label{obs}

In addition to their utility in proving Borel complete reducts, the theories $T_h$ in Section 2 illustrate some novel behaviors.
First off, model theoretically, these theories are extremely simple.  More precisely, each  theory $T_h$ is  weakly minimal with the geometry of every strong type trivial
(such theories are known as mutually algebraic in \cite{MA}). 

Additionally, the theories $T_h$ are the simplest known examples of theories that are Borel complete, but not $\lambda$-Borel complete for all cardinals $\lambda$.
For $\lambda$ any infinite cardinal, $\lambda$-Borel completeness was introduced in \cite{LSh}.  Instead of looking at $L$-structures with universe $\omega$, we consider
$X_L^\lambda$, the set of $L$-structures with universe $\lambda$.  We topologize $X_L^\lambda$ analogously; namely a basis consists of all sets
$$U_{\phi(\alpha_1,\dots,\alpha_n)}:=\{M\in X_L^\lambda:M\models \phi(\alpha_1,\dots,\alpha_n)\}$$
for all $L$-formulas $\phi(x_1,\dots,x_n)$ and all $(\alpha_1,\dots,\alpha_n)\in\lambda^n$.  Define a subset of $X_L^\lambda$ to be {\em $\lambda$-Borel} if it is is the smallest
$\lambda^+$-algebra containing the basic open sets, and call a function $f:X_{L_1}^\lambda\rightarrow X_{L_2}^\lambda$ to be {\em $\lambda$-Borel} if the inverse image of 
every basic open set is $\lambda$-Borel.  For $T,S$ theories in languages $L_1,L_2$, respectively we say that $\Mod_\lambda(T)$ is {\em $\lambda$-Borel reducible} to
$\Mod_\lambda(S)$ if there is a $\lambda$-Borel $f:\Mod_\lambda(T)\rightarrow \Mod_\lambda(S)$ preserving back-and-forth equivalence in both directions (i.e.,
$M\equiv_{\infty,\omega} N\Leftrightarrow f(M)\equiv_{\infty,\omega} f(N)$).  

As back-and-forth equivalence is the same as isomorphism for countable structures, $\lambda$-Borel reducibility when $\lambda=\omega$ is identical to Borel reducibility.
As before, for any infinite $\lambda$, there is a maximal class under $\lambda$-Borel reducibility, and we say a theory is {\em $\lambda$-Borel complete} if it is in this maximal class.
All of the `classical' Borel complete theories, e.g., graphs, linear orders, groups, and fields, are $\lambda$-Borel complete for all $\lambda$.  However, the theories $T_h$ are not.

\begin{Lemma}  If $T$ is mutually algebraic in a countable language, then there are at most $\beth_2$ pairwise $\equiv_{\infty,\omega}$-inequivalent models (of any size).
\end{Lemma}

\begin{proof}  We show that every model $M$ has an $(\infty,\omega)$-elementary substructure of size $2^{\aleph_0}$, which suffices.  
So, fix $M$ and choose an arbitrary countable $M_0\preceq M$.  By Propositon~4.4 of \cite{MA},  $M\setminus M_0$  can be decomposed into countable components,
and any permutation of isomorphic components induces an automorphism of $M$ fixing $M_0$ pointwise.
As there are at most $2^{\aleph_0}$ non-isomorphic components over $M_0$, choose a substructure $N\subseteq M$ containing $M_0$ and, for each isomorphism type of
a component, $N$ contains either all of copies in $M$ (if there are only finitely many) or else precisely $\aleph_0$ copies if $M$ contains infinitely many copies.
It is easily checked that $N\preceq_{\infty,\omega} M$.
\end{proof}

\begin{Corollary}  No mutually algebraic theory $T$ in a countable language is $\lambda$-Borel complete for $\lambda\ge \beth_2$.  In particular, $T_h$ is Borel complete, but not $\lambda$-Borel complete for large $\lambda$.
\end{Corollary}

\begin{proof}  Fix $\lambda\ge \beth_2$.  It is readily checked that there is a family of  $2^{\lambda}$ graphs that are pairwise not back and forth equivalent.
As there are fewer than $2^{\lambda}$ $\equiv_{\infty,\omega}$-classes of models of $T$, there cannot be a $\lambda$-Borel reduction of graphs into $\Mod_\lambda(T)$.
\end{proof}
 
 In \cite{URL}, another example of a Borel complete theory that is not $\lambda$-Borel complete for all $\lambda$ is given (it is dubbed $TK$ there) but the $T_h$ examples
 are cleaner.  In order to understand this behavior, in \cite{URL} we call a theory $T$ {\em grounded} if every potential canonical Scott sentence $\sigma$ of a model of $T$ (i.e.,
 in some forcing extension $\V[G]$ of $\V$, $\sigma$ is a canonical Scott sentence of some model, then $\sigma$ is a canonical Scott sentence of a model in $\V$.
Proposition~5.1 of  \cite{URL} proves that every theory of refining equivalence relations is grounded.  By contrast, we have

\begin{Proposition}  If $T$ is Borel complete with a cardinal bound on the number of $\equiv_{\infty,\omega}$-classes of models,
then $T$ is not grounded.  In particular, $T_h$ is not grounded.
\end{Proposition}

\begin{proof}  Let $\kappa$ denote the number of $\equiv_{\infty,\omega}$-classes of models of $T$.  If $T$ were grounded, then $\kappa$ would also bound the number of potential canonical Scott sentences.  As the class of graphs has a proper class of potential canonical Scott sentences, it would follow from Theorem~3.10 of \cite{URL} that $T$ could not  be Borel complete.\end{proof}

\end{document}